\documentclass[proceedings,submission%
]{dmtcs} 

\usepackage[latin1]{inputenc}
\usepackage{subfigure}

\usepackage{tikz}
	\usetikzlibrary{arrows,backgrounds,decorations,positioning,fit,petri}
\usepackage{ntheorem}
\usepackage{amsmath}

\usepackage{amsfonts}
\usepackage{amscd}

\usepackage[round]{natbib}

\newcommand{\be}{\begin{equation}}
\newcommand{\ee}{\end{equation}}
\newcommand{\beq}{\begin{equation}}
\newcommand{\eeq}{\end{equation}}
\newcommand{\bea}{\begin{eqnarray}}
\newcommand{\eea}{\end{eqnarray}}
\def\beqa{\begin{eqnarray}}
\def\eeqa{\end{eqnarray}}

\def\KK{{\mathbb K}}

\def\CC{{\mathbb C}}
\def\RR{{\mathbb R}}

\newcommand{\eq}{\begin{equation}}
\newcommand{\eqa}{\begin{eqnarray}}
\newcommand{\en}{\end{equation}}
\newcommand{\ena}{\end{eqnarray}}

\newtheorem{theorem}{Theorem}[section]

\newtheorem{corollary}[theorem]{Corollary}

\newtheorem{definition}[theorem]{Definition}

\newtheorem{lemma}[theorem]{Lemma}

\newtheorem{proposition}[theorem]{Proposition}
\newtheorem{remark}[theorem]{Remark}

\def\I{\mathcal{I}}

\def\KK{{\mathbb K}}

\def\CC{{\mathbb C}}
\def\RR{{\mathbb R}}

\def\beqa{\begin{eqnarray}}
\def\eeqa{\end{eqnarray}}

\author[G. Duchamp, N. Hoang-Nghia, T. Krajewski  and A. Tanasa]
{G. Duchamp\addressmark{1}
\thanks{Partially supported by
    a Univ. Paris 13 Sorbonne Paris Cit\'e BQR grant}, 
N. Hoang-Nghia\addressmark{1}
\thanks{Partially supported by
    a Univ. Paris 13 Sorbonne Paris Cit\'e BQR grant}, 
T. Krajewski\addressmark{2}\thanks{Partially supported by
    a Univ. Paris 13 Sorbonne Paris Cit\'e BQR grant} \and  
  A. Tanasa\addressmark{1,3}\thanks{Partially supported by
    a Univ. Paris 13 Sorbonne Paris Cit\'e BQR grant and by grants PN 09 37 01
02 and CNCSIS Tinere Echipe 77/04.08.2010}}

\title[Renormalization group-like proof of the universality of the Tutte polynomial for matroids]
{Renormalization group-like proof of the universality of the Tutte polynomial for matroids}

\address{\addressmark{1}
LIPN, UMR 7030 CNRS, Institut Galil\'ee - 
Universit\'e Paris 13, Sorbonne Paris Cit\'e, 
99 av. J.-B. Cl\'ement, 93430 Villetaneuse, France\\
\addressmark{2}
CPT, 
CNRS UMR 7332, Aix Marseille Universit\'e, 
Campus de Luminy, Case 907, 13288 Marseille cedex 9, France\\
\addressmark{3}
Horia Hulubei National Institute for Physics and Nuclear Engineering,
P.O.B. MG-6, 077125 Magurele, Romania
\\
}

\begin{document}
\maketitle
\begin{abstract}
\paragraph{Abstract}
In this paper we give a new proof of the universality of the Tutte polynomial
 for matroids. This proof uses appropriate characters of Hopf algebra of matroids, 
algebra introduced by Schmitt (1994). We show that these 
Hopf algebra characters are solutions of 
some differential equations which are of the same type as the differential equations 
used to describe the renormalization group flow in quantum field theory.
 This approach allows us to also prove, in a different way,
a matroid Tutte polynomial convolution formula published by Kook, Reiner and Stanton (1999). 
This FPSAC contribution is an extended abstract.

\paragraph{R\'esum\'e.}
Dans cet article, nous donnons une nouvelle preuve de l'universalit\'e du polyn\^ome 
de Tutte
  pour les matro\"ides. Cette preuve utilise des caract\`eres 
appropri\'es de l'alg\`ebre de Hopf des matro\"ides
introduite par Schmitt (1994). Nous montrons que ces
caract\`eres alg\`ebre de Hopf sont des solutions de
des \'equations diff\'erentielles du m\^eme 
type que les \'equations diff\'erentielles
utilis\'ees pour d\'ecrire le flux du 
groupe de renormalisation en th\'eorie quantique de champs. 
Cette approche nous permet aussi de d\'emontrer, d'une mani\`ere diff\'erente, 
une formule de convolution du polyn\^ome de Tutte des matro\"ides, 
formule publi\'ee par Kook, Reiner et Stanton (1999).
Cette contribution FPSAC est un r\'esum\'e \'etendu. 
\end{abstract}

\section{Introduction}
\label{sec:in}

The interplay between algebraic combinatorics and quantum field theory (QFT) has become 
more and more present within the spectrum of Combinatorial Physics 
(spectrum represented by many other subjects, such as
 the combinatorics of quantum mechanics, of statistical physics 
or of integrable systems - see, for example,
 \cite{bf}, \cite{5}, the review article \cite{dflast}, the 
Habilitation \cite{io-hdr} and references within).

One of the most known results lying at this frontier between 
algebraic combinatorics and QFT is the celebrated 
Hopf algebra \cite{ck}, 
describing the combinatorics of renormalization in QFT.
It is worth emphasizing that the coproduct of this 
type of Hopf algebra is based on a selection/contraction rule
 (one has on the coproduct left hand side (lhs) some selection of
 a subpart of the entity the coproduct acts on, while
 on the coproduct right hand side (rhs) one has the result of the
{\it contraction} of the selected subpart).
This type of rule 
(appearing in other situations in Mathematical Physics, 
see also \cite{fab, io-dirk, mar, io-sf}) 
is manifestly distinct from the 
selection/deletion one, largely studied in 
algebraic combinatorics (see, for example, 
\cite{thi1} and references within).

In this paper, we use characters of the matroid Hopf algebra introduced in \cite{s}
to prove the universality property of the Tutte polynomial for matroids.
We use a Combinatorial Physics approach, namely we use a renormalization group-like 
differential equation to prove the respective recipe theorem. Our method also allows 
to give a new proof of a matroid Tutte polynomial convolution formula 
given in \cite{reiner}. This approach generalizes the one given in \cite{km} for 
the universality of the Tutte polynomial for graphs.
Moreover, the demonstrations we give here allow us to also have proofs
of  the graph results conjectured  
in  \cite{km}.

\medskip

The paper is structured as follows. In the following section we briefly present the 
renormalization group flow equation and show that equations of such type have already been 
successfully used in combinatorics. The third section defines matroids, as well the Tutte polynomial 
for matroids and the matroid Hopf algebra defined in \cite{s}.
The following section defines some particular infinitesimal characters of this Hopf algebra
 as well as their exponential - proven, later on, to be non-trivially 
related to the Tutte polynomial for matroids. The fifth section uses all this tools to 
give a new proof of the matroid Tutte polynomial convolution formula 
given in \cite{reiner} and of the recipe theorem for the Tutte polynomial of matroids.

\section{Renormalization group in quantum field theory - a glimpse}
\renewcommand{\theequation}{\thesection.\arabic{equation}}
\setcounter{equation}{0}
\label{sec:RG}

A QFT model (for a general introduction 
to QFT and not just an introduction, see for example the book \cite{book-ZJ}) 
is defined by means of a functional integral of the exponential of an 
{\it action} $S$ which, from a mathematical point of view, is a functional
 of the {\it fields} of the model. For the $\Phi^4$ scalar model - the simplest QFT model - 
there is only one type of
field, which we denote by $\Phi (x)$. From a mathematical point of view,
for an Euclidean QFT scalar model,
the field $\Phi(x)$ is a function, $\Phi:\RR^D\to \KK$, 
where $D$ is usually taken equal to $4$ (the dimension of the space) 
and $\KK\in\{\RR,\CC\}$ (real, respectively complex fields).

The quantities computed in QFT are generally divergent. One thus has to consider a 
real, positive, {\it cut-off}
 $\Lambda$ - the flowing parameter.
 This leads to a family of cut-off dependent actions, family denoted by
$S_\Lambda$. The derivation $\Lambda\frac{\partial S_\Lambda}{\partial\Lambda}$ gives the 
{\it renormalization group equation}.

The quadratic part of the action - the {\it propagator} of the model - can be written in the following 
way
\beqa
\label{propa}
C_{\Lambda,\Lambda_0}(p,q)=\delta(p-q)\int_{\frac{1}{\Lambda_0}}^{\frac{1}{\Lambda}}
d\alpha e^{-\alpha p^2},
\eeqa
with $p$ and $q$ living in the Fourier transformed space $\RR^D$ and 
$\Lambda_0$ a second real, positive cut-off. In perturbative QFT, one has to 
consider {\it Feynman graphs}, and to associate to each such a graph a 
{\it Feynman integral}  
(further related to quantities actually measured in physical experiments).
The contribution of an edge of such a Feynman graph to its associated Feynman integral 
is given by an integral such as \eqref{propa}.

One can then get (see \cite{pol}) 
the Polchinski flow equation
\beqa
\label{eq:pol}
\Lambda \frac{\partial S_\Lambda}{\partial\Lambda}=
\int_{\RR^{2D}} \frac 12 d^D p d^D q
\Lambda\frac{\partial C_{\Lambda,\Lambda_0}}{\partial \Lambda}
\left(
\frac{\delta^2S}{\delta \tilde \Phi (p)\delta \tilde \Phi (q)}
-\frac{\delta S}{\delta \tilde \Phi (p)} \frac{\delta S}{\delta \tilde \Phi (q)}
\right),
\eeqa
where $\tilde \Phi$ represents the Fourier transform of the function $\Phi$. 
The first term in the right hand side (rhs) of the equation above corresponds 
to the derivation of a propagator associated to a bridge 
in the respective Feynman graph.
The second term corresponds to an edge which is not a bridge and is 
part of some circuit in the graph. One can see this diagrammatically in Fig. \ref{fig:pol}.
\begin{figure}[h]
\centerline{\includegraphics[width=12cm]{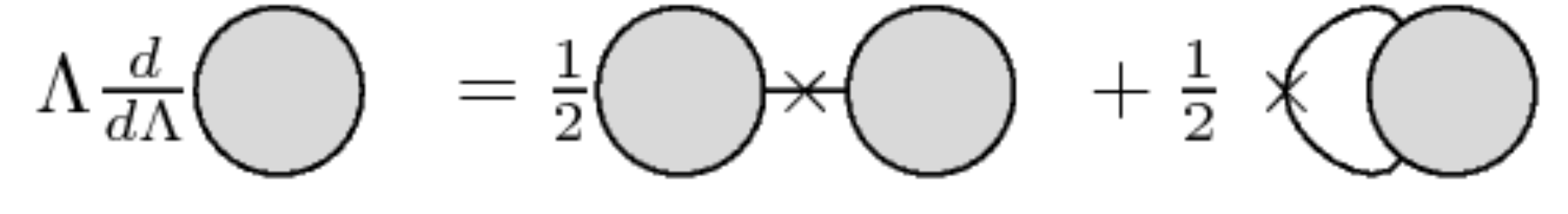}   }
\caption{Diagrammatic representation of the flow equation.}
\label{fig:pol}
\end{figure}

This equation can then be used to prove perturbative renormalizability in QFT.
Let us also stress here, that an equation of this type is also used 
to prove a result of E. M. Wright which expresses the generating function 
of connected graphs under certain conditions (fixed excess). 
To get this generating 
functional (see, for example, Proposition $II.6$ the book \cite{book-fs}), 
one needs to consider contributions of two types of edges (first contribution
when  the 
edge is a bridge and a second one when not - see again Fig. \ref{fig:pol}).

As already announced in the Introduction, we will use such an equation to prove the 
universality of the matroid Tutte polynomial (see section \ref{sec:proof}).

\section{Matroids: the Tutte polynomial and the Hopf algebra}
\renewcommand{\theequation}{\thesection.\arabic{equation}}
\setcounter{equation}{0}
\label{sec:hopfm}

In this section we recall the definition and some properties of the
 Tutte polynomial for matroids as well as of the matroid Hopf algebra 
defined in \cite{s}.

Following the book \cite{oxley}, one has the following definitions:

\begin{definition}
A {\bf matroid} M is a pair $(E, \mathcal{I})$ consisting of a finite set E 
and a collection of subsets of E satisfying:
\begin{itemize}
\item[(I1)]$\mathcal{I}$ is non-empty. 
\item[(I2)] Every subset of every member of $\mathcal{I}$ is also in $\mathcal{I}$.
\item[(I3)] If $X$ and $Y$ are in $\mathcal{I}$ and $|X| = |Y | + 1$, then there is an element 
$x$ in $X - Y$ such that $Y \cup \{x\}$ is in $\mathcal{I}$.
\end{itemize}
The set $E$ is the ground set of the matroid and the members of 
$\mathcal{I}$ are the independent sets of the matroid.
\end{definition}

One can associate matroid to graphs - graphic matroids. Nevertheless, not all matroids 
are graphic matroids.

Let $E$ be an $n-$element set and let $\I$ be the collection of subsets 
of $E$ with at most $r$ elements, $0\le r\le n$. One can check that $(E,\I)$ 
is a matroid; it is called the {\bf uniform matroid} $U_{r,n}$.

\begin{remark}
\label{remarca}
 If one takes $n=1$, there are only two matroids, namely $U_{0,1}$ and $U_{1,1}$ and both of these 
matroids are graphic matroids. The graphs these two 
matroids correspond to are the graphs with one edge of Fig. \ref{graf1} and Fig. \ref{graf2}.
\begin{figure}[h]
\centerline{\includegraphics[width=1cm]{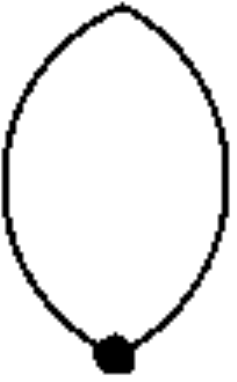}   }
\caption{The graph corresponding to the matroid $U_{0,1}$.}
\label{graf1}
\end{figure}
\begin{figure}[h]
\centerline{\includegraphics[width=2cm]{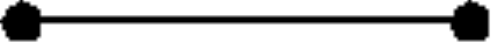}   }
\caption{The graph corresponding to the matroid $U_{1,1}$.}
\label{graf2}
\end{figure}
In the first case, the edge is a loop (in graph theoretical terminology) or a tadpole 
(in QFT language). In the second case, the edge represents a bridge (in graph theoretical 
terminology) or a $1$-particle-reducible line (in QFT terminology) - the number of connected 
components of the graphs increases by $1$ if one deletes the respective edge. In matroid 
terminology, these two particular cases correspond 
to a loop and respectively to a coloop (see Definitions 
\ref{def-loop} and respectively \ref{def-coloop} below).
\end{remark}

\begin{definition}
 The collection of maximal independent sets of a matroid are called bases. 
The collection of minimal dependent sets of a matroid are called circuits.
\end{definition}

Let $M=(E,\I)$ be a matroid and let $\cal B=\{ B\}$ be the collection of bases of $M$. 
Let ${\cal B}^\star = \{E - B: B \in {\cal B} \}$. 
Then ${\cal B}^\star$ is the collection of bases of a matroid $M^\star$ on E. 
The matroid $M^*$ is called the dual of $M$.

\begin{definition}
\label{def-rank}
Let $M=(E,\I)$ be a matroid. 
The {\bf rank} $r(A)$ of $A \subset E$ is defined as the cardinal of a maximal independent set in $A$.
\begin{equation}\label{eq:rankfunc}
r(A) = max\{|B| \mbox{  s.t.  } B \in \I, B \subset A\}\ .
\end{equation}
\end{definition}

\begin{definition}
\label{def-nullity}
Let $M=(E,\mathcal{I})$ be a matroid with a ground set $E$. 
The {\bf nullity} function is given by  
\begin{equation}
\label{defnul}
n(M) = |E|- r(M)\ .
\end{equation}
\end{definition}

\begin{definition}
\label{def-loop}
Let $M=(E,\I)$ be a matroid.
The element $e \in E$. 
 is a \textbf{loop} iff $\{e\}$ is the circuit.
\end{definition}

\begin{definition}
\label{def-coloop}
Let $M=(E,\I)$ be a matroid. 
The element $e\in E$ is a \textbf{coloop} iff, for any basis $B$, 
$e\in B$ .
\end{definition}

Let us now define two basic operations on matroids.
Let $M$ be a matroid $(E,\I$) and $T$ be a subset of $E$. 
Let $\I'=\{I\subseteq E-T: I \in \I\}$. 
One can check that $(E-T,\I')$ is a matroid. 
We denote this matroid by $M\backslash T$ - the {\bf deletion} of $T$ from $M$.
The {\bf contraction} of $T$ from $M$, $M/T$, is given by the formula: 
$M/T=(M^\star\backslash T)^\star$.

Let us also recall the following results:

\begin{lemma}\label{lm:res-del}
Let $M$ be a matroid $(E,\I$) and $T$ be a subset of $E$. One has \begin{equation}
M|_T = M\backslash_{E-T} \ .
\end{equation}
\end{lemma}

\begin{lemma}
\label{lm:coloop}
If $e$ is a coloop of a matroid $M=(E,\I)$, then $M/e = M\backslash e$.
\end{lemma}

\begin{lemma}
\label{lm:rankcontra}
Let $M=(E,\I)$ be a matroid and $T\subseteq E$, then, for all $X \subseteq E-T$, \begin{equation}
r_{M/T}(X) = r_M(X\cup T) - r_M(T)\ .
\end{equation}
\end{lemma}

Let us now define the Tutte polynomial for matroids:

\begin{definition}
Let $M=(E,\mathcal{I})$ be a matroid. The {\bf Tutte polynomial} is given by the following formula:
\begin{equation}
\label{def-tutte}
T_M(x,y) = \sum_{A \subseteq E} (x-1)^{r(E) - r(A)} (y-1)^{n(A)}.
\end{equation}
The sum is computed over all subset of the matroid's ground set. 
\end{definition}

\begin{definition}
Let $\psi$ be the matroid duality map,
that is a map associating to any matroid $M$ its dual,
$\psi(M) = M^\star$. 
\end{definition}

It is worth stressing here that one can define the dual of any matroid; this is not the case for graphs, 
where only the dual of planar graph can be defined.

Let us recall, from \cite{bo} that
\beqa
\label{reltdual}
T_M (x,y)=T_{M^\star}(y,x).
\eeqa

In \cite{s}, 
as a particularization of a more general construction of incidence Hopf algebras, the following result 
was proved:

\begin{proposition}
\label{prop-cs}
If $\mathcal{M}$ is a minor-closed family of matroids 
then $k(\widetilde{\mathcal{M}})$ is a coalgebra, 
with coproduct $\Delta$ and counit $\epsilon$ 
determined by 
\beqa
\label{defc}
\Delta(M) = \sum_{A\subseteq E} M|A \otimes M/A
\eeqa
and $\epsilon(M) = \begin{cases} 1, \mbox{ if } E = \emptyset, \\ 0 \mbox{ otherwise ,} \end{cases}$ for all $M = (E,\mathcal{I}) \in \mathcal{M}$. If, furthermore, the family $\mathcal{M}$ is closed under formation of direct sums, then $k(\widetilde{\mathcal{M}})$ is a Hopf algebra, with product induced by direct sum.
\end{proposition}

We refer to this Hopf algebra as the matroid Hopf algebra. We follow \cite{cs}
and, by a slight abuse of notation, we denote in the same way a matroid and its isomorphic class,
since the distinction will be clear from the context (as it is already in 
Proposition \ref{prop-cs}).

We denote the unit of this Hopf algebra by $\mathbf{1}$ (the empty matroid, or $U_{0,0}$).

\section{Characters of matroid Hopf algebra}
\renewcommand{\theequation}{\thesection.\arabic{equation}}
\setcounter{equation}{0}
\label{sec:ch}




Let us give the following definitions:

\begin{definition}
Let $f,g$ be two mappings in $Hom(\mathcal{M},\mathcal{M})$. 
The convolution product of $f$ and $g$ is given by the following formula
\begin{equation}
\label{def-conv}
f\ast g = m\circ (f\otimes g)\circ \Delta,
\end{equation}
where $m$ denotes the Hopf algebra multiplication, given here by direct sum (see above).
\end{definition}

\begin{definition}
A matroid Hopf algebra {\bf character} $f$ is an algebra morphism 
from 
the matroid Hopf algebra
into a fixed commutative ring $\mathbb{K}$, 
such that 
\beqa
f(M_1\oplus M_2) = f(M_1)f(M_2),\ \  f(\mathbf{1}) =1_\KK.
\eeqa
\end{definition}

\begin{definition}
A matroid Hopf algebra {\bf infinitesimal 
character} $g$ is an algebra morphism from 
the matroid Hopf algebra 
into a fixed commutative ring $\KK$, 
such that 
\beqa
g(M_1\oplus M_2) = f(M_1)\epsilon(M_2)+\epsilon(M_1)g(M_2).
\eeqa
\end{definition}

Since we work in a Hopf algebra where the non-trivial part of the 
coproduct is nilpotent, we can also define an exponential map by
the following expression 
\beqa
\label{def-expstar}
\mbox{exp}_\ast(\delta)=\epsilon+\delta+\frac 12 \delta\ast\delta +\ldots
\eeqa
where $\delta$ is an infinitesimal character.

As already stated above (see Remark \ref{remarca}),  
there are only two matroids with unit cardinal ground set,
$U_{0,1}$ and $U_{1,1}$. We now define two 
maps $\delta_{\mathrm{loop}}$ and $\delta_{\mathrm{coloop}}$.


\beqa
\label{def-dloop}
\delta_{\mathrm{loop}} (M) = \begin{cases}
1_\KK \mbox{ if } M 
= U_{0,1},\\
0_\KK \mbox{ otherwise }.
\end{cases}
\eeqa

\beqa
\label{def-dtree}
\delta_{\mathrm{coloop}} (M) = \begin{cases}
1_\KK \mbox{ if } M 
= U_{1,1},\\
0_\KK \mbox{ otherwise }.
\end{cases}
\eeqa

One can directly check that these maps are {\it infinitesimal characters} of the matroid 
Hopf algebra defined above. 

\medskip

One then has:

\begin{lemma}
\label{lema-help}
 Let $M=(E,\mathcal{I})$ be a matroid. One has 
\begin{equation}
\label{help}
\mbox{exp}_*\{a\delta_{\mathrm{coloop}}+b\delta_{\mathrm{loop}}\}(M) = a^{r(M)}b^{n(M)}.
\end{equation} 
\end{lemma}
\begin{proof}
Using the definition \eqref{def-expstar}, the lhs of the identity \eqref{help} above
writes:
\begin{equation}
 \left(\sum_{k=0}^\infty \frac{(a\delta_{\mathrm{coloop}}+b\delta_{\mathrm{loop}})^k}{k!}\right)(M).
\end{equation}
All the terms in the sum above vanish, except the one for whom 
$k$ is equal to $|E|$. Using the definition \eqref{def-conv} of the convolution product,
this term writes
\beqa
\frac{1}{k!}\left( \sum_{i=0}^k a^{k-i}b^{i}
\sum_{\substack{i_1,\dots,i_n \\ j_1,\dots, j_m \\ i_1+ \dots + i_n =k - i \\ j_1+ \dots + j_m = i}}
\delta_{\mathrm{coloop}}^{\otimes(i_1)}\otimes\delta_{\mathrm{loop}}^{\otimes(j_1)}
\otimes\dots\otimes \delta_{\mathrm{coloop}}^{\otimes(i_n)}
\otimes\delta_{\mathrm{loop}}^{\otimes(j_m)}\right) 
\left(\sum_{ (i)} M^{(1)}\otimes\dots\otimes M^{(k)}\right),
\eeqa
where we have used the notation $\Delta^{(k-1)}(M)=\sum_{ (i)} M^{(1)}\otimes\dots\otimes M^{(k)}$.
Using the definitions \eqref{def-dloop} and respectively \eqref{def-dtree} 
of the infinitesimal characters
$\delta_{\mathrm{loop}}$ and respectively $\delta_{\mathrm{coloop}}$, implies that
the submatroids $M^{(j)}$ ($j=1,\ldots,k$) are equal to $U_{0,1}$ or $U_{1,1}$.

Using the definition of the nullity and of the rank of a matroid concludes the proof.
\end{proof}


We now define the following map:

\beqa
\label{def-alpha}
\alpha(x,y,s,M) := \mbox{exp}_\ast s\{\delta_{\mathrm{coloop}}
+(y-1)\delta_{\mathrm{loop}}\}\ast \mbox{exp}_\ast s\{(x-1)\delta_{\mathrm{coloop}}+\delta_{\mathrm{loop}}\}
(M).
\eeqa

One then has:

\begin{proposition}
 The map \eqref{def-alpha} is a character.
\end{proposition}
\begin{proof}
The proof can be done by a direct check. On a more general 
basis, this is a consequence of the fact that 
$\delta_{\mathrm{loop}}$ and $\delta_{\mathrm{coloop}}$ are infinitesimal characters
and the space of infinitesimal characters is a vector space; 
thus 
$s\{\delta_{\mathrm{coloop}}
+(y-1)\delta_{\mathrm{loop}}\}$ and 
$s\{(x-1)\delta_{\mathrm{coloop}}+\delta_{\mathrm{loop}}\}$ 
are
infinitesimal characters.
Since, $\mbox{exp}_\ast (h)$ is a character when $h$ is an infinitesimal character
and since the convolution of two characters is a character,
one gets that $\alpha$ is a character.
\end{proof}

Let us define a map
\begin{equation}
\label{defphi}
\varphi_{a,b} (M) \longmapsto a^{r(M)}b^{n(M)}M \ .
\end{equation} 

\begin{lemma}
\label{lemamor}
The map $\varphi_{a,b}$ is a bialgebra automorphism.
\end{lemma}
\begin{proof}
One can directly check that the map $\varphi_{a,b}$ is an algebra automorphism.
Let us now check that this map is also a coalgebra automorphism. Using 
Lemma \ref{lm:res-del} and 
Lemma \ref{lm:rankcontra}, 
\beqa
\label{ranks}
r(M|_T) + r(M/T) = r(M).
\eeqa
Thus, using  the definitions of the map $\varphi_{a,b}$ of the matroid coproduct, one has:
\begin{equation}
 \Delta\circ \varphi_{a,b}(M) =
\sum_{T\subseteq E} (a^{r(M|_T)}b^{n(M|_T)}M|_T) \otimes (a^{r(M/_T)}b^{n(M/_T)}M/_T).
\end{equation}
Using again the definition of the map $\varphi_{a,b}$ leads to
\begin{equation}
 \Delta\circ \varphi_{a,b}(M) =(\varphi_{a,b}\otimes \varphi_{a,b})\circ \Delta (M),
\end{equation}
which concludes the proof.
\end{proof}

\section{Proof of the universality of the Tutte polynomial for matroids}
\renewcommand{\theequation}{\thesection.\arabic{equation}}
\setcounter{equation}{0}
\label{sec:proof}



Let $M=(E,\mathcal{I})$ be a matroid. 
One has:
\beqa
\label{pre-reiner}
\alpha(x,y,s,M) &=& 
\mbox{exp}_{\ast}
\left(s(\delta_{\mathrm{coloop}}+(y-1)\delta_{\mathrm{loop}})\right)\ast\mbox{exp}_{\ast}\left(s(-\delta_{\mathrm{coloop}}+\delta_{\mathrm{loop}})\right)
\nonumber\\
&\ast&\mbox{exp}_{\ast}\left(s(\delta_{\mathrm{coloop}}-\delta_{\mathrm{loop}})\right)\ast
\mbox{exp}_{\ast}\left(s((x-1)\delta_{\mathrm{coloop}}+\delta_{\mathrm{loop}})\right).
\eeqa 

\begin{proposition}
\label{alpha-tutte}
Let $M=(E,\mathcal{I})$ be a matroid. 
The character $\alpha$ is related to the Tutte polynomial of matroids by the 
following identity:
\begin{equation}
\alpha(x,y,s,M) = s^{|E|}T_M(x,y)\ .
\end{equation} 
\end{proposition}
\begin{proof}
 Using the definition \eqref{def-conv} of the convolution product in the definition 
\eqref{def-alpha} of the character $\alpha$, one has the following identity:
\begin{align}
\label{eq:dev2mat} 
\alpha(x,y,s,M) = \sum_{A \subseteq E} \mbox{exp}_*s\{\delta_{\mathrm{coloop}}+(y-1)\delta_{\mathrm{loop}}\}(M|_A)\ 
\mbox{exp}_*s\{(x-1)\delta_{\mathrm{coloop}}+\delta_{\mathrm{loop}}\}(M/A).
\end{align}
We can now apply Lemma \ref{lema-help} on each of the two terms in the rhs of equation 
\eqref{eq:dev2mat} above. This leads to the result.
\end{proof}

Using \eqref{reltdual} and the Proposition \ref{alpha-tutte}, one has:

\begin{corollary}
 One has:
\beqa
\alpha(x,y,s,M) = \alpha(y,x,s,M^\star).
\eeqa
\end{corollary}

This allows to give a different proof of a matroid Tutte polynomial
convolution identity, which was shown in \cite{reiner}.
One has:

\begin{corollary}
(Theorem $1$ of \cite{reiner})
 The Tutte polynomial satisfies
\begin{equation}
 T_M(x,y)=\sum_{A\subset E} T_{M|A}(0,y) T_{M/A}(x,0).
\end{equation}
\end{corollary}
\begin{proof}
Taking $s=1$, this is as a direct consequence of 
identity \eqref{pre-reiner}, 
and of Proposition \ref{alpha-tutte}. 
\end{proof}

Let us now define:
\beqa
[f,g]_\ast :=f\ast g - g\ast f.
\eeqa
Using the definition \eqref{def-alpha} of the Hopf algebra character $\alpha$, 
one can directly prove the following result:

\begin{proposition}
 \label{propalpha}
The character $\alpha$ is the solution of the differential equation:
 \begin{equation}
\label{diffeqalpha}
\frac{d\alpha}{ds} = x\alpha \ast \delta_{\mathrm{coloop}} + y\delta_{\mathrm{loop}}\ast\alpha + 
\left[\delta_{\mathrm{coloop}},\alpha\right]_\ast - \left[\delta_{\mathrm{loop}},\alpha\right]_\ast \ .
\end{equation}
\end{proposition}

It is the fact that the matroid Tutte polynomial is a solution of the differential equation 
\eqref{diffeqalpha} that will be used now to prove the 
universality of the matroid Tutte polynomial. In order to do that, we take a
four-variable matroid polynomial $Q_M(x,y,a,b)$ satisfying 
a multiplicative law and which has the following properties: 
\begin{itemize}
 \item if $e$ is a coloop, then
\begin{equation}
Q_M(x,y,a,b) = xQ_{M\backslash e}(x,y,a,b)\ ,
\end{equation}
\item  if $e$ is a loop, then
\begin{equation}
 Q_M(x,y,a,b) = yQ_{M/e}(x,y,a,b)
\end{equation}
\item if $e$ is a nonseparating point, then
\beqa
Q_M(x,y,a,b) = aQ_{M\backslash e}(x,y,a,b)+bQ_{M/e}(x,y,a,b).
\eeqa
\end{itemize}

\begin{remark}
 Note that, when one deals with the same problem in the case of graphs, 
a supplementary multiplicative condition for the case of one-point joint of two graphs 
({\it i. e.} identifying a vertex of the first graph and a vertex of the second graph 
  into a single vertex of the resulting graph) 
is required (see, for example, \cite{em} or \cite{sokal}).
\end{remark}

We now 
define the map:
\beqa
\label{defbeta}
\beta(x,y,a,b,s,M):=s^{|E|}Q_M(x,y,a,b).
\eeqa
One then directly check (using the definition \eqref{defbeta} above and 
the multiplicative property of the polynomial $Q$)  
that 
this map is again a matroid Hopf algebra character.

\begin{proposition}
\label{propbeta}
The character \eqref{defbeta} satisfies the following differential equation: 
\begin{equation}
\label{diffeqbeta}
\frac{d\beta}{ds} (M) = \left( x\beta \ast \delta_{\mathrm{coloop}} + y\delta_{\mathrm{loop}}\ast \beta + 
b [\delta_{\mathrm{coloop}},\beta]_\ast - a [\delta_{\mathrm{loop}},\beta]_\ast\right) (M).
\end{equation}
\end{proposition}
\begin{proof}
Applying the definition \eqref{def-conv} of the convolution product,
the rhs of equation \eqref{diffeqbeta} above writes
 \begin{align}
\label{eq:betamat}
& = (x-b)\sum_{A\subseteq E} \beta(M|_A) \delta_{\mathrm{coloop}}(M/_A) + (y-a)\sum_{A\subseteq E} \delta_{\mathrm{loop}}(M|_A)\beta(M/_A)\notag\\
& + b \sum_{A\subseteq E} \delta_{\mathrm{coloop}}(M|_A)\beta(M/_A) + a \sum_{A\subseteq E} \beta(M|_A) \delta_{\mathrm{loop}}(M/_A).
\end{align}
Using the definitions \eqref{def-dloop} and respectively 
\eqref{def-dtree} of the infinitesimal characters
$\delta_{\mathrm{loop}}$ and respectively $\delta_{\mathrm{coloop}}$, constraints the sums on the subsets $A$ above. 
The rhs of \eqref{diffeqbeta}
becomes:
\beqa
& (x-b)\sum_{A, M/_A = U_{1,1}} \beta(M|_A) + (y-a)\sum_{A, M|_A= U_{0,1}} \beta(M/_A)\notag\\
& + b \sum_{A, M|_A = U_{1,1}} \beta(M/_A) + a \sum_{A, M/_A = U_{0,1}}\beta(M|_A)
\eeqa
We now apply the definition of the Hopf algebra character $\beta$; one obtains:
\beqa
\label{int1}
&s^{|E|-1}[(x-b)\sum_{A, M/_A = U_{1,1}} Q(x,y,a,b,M|_A) + 
(y-a)\sum_{A, M|_A= U_{0,1}} Q(x,y,a,b,M/_A)\notag\\
& + b \sum_{A, M|_A = U_{1,1}} Q(x,y,a,b,M/_A) + a \sum_{A, M/_A = U_{0,1}}Q(x,y,a,b,M|_A)].
\eeqa
We can now directly analyze the four particular cases $M/_A=U_{1,1}$, $M/_A=U_{0,1}$, 
$M|_A=U_{1,1}$ and $M|_A=U_{0,1}$:
\begin{itemize}
 \item 
If 
 $M/_A=U_{1,1}$, we can denote the ground set of $M/_A$ by $\{ e \}$. Note that $e$ is a coloop.
From the Lemma \ref{lm:res-del}, one has $M|_A = M\backslash_{E-A}=M\backslash e$.
One then has 
$Q(x,y,a,b,M) = xQ(x,y,a,b,M|_A)$.

\item
If $M|_A= U_{0,1}$, then $A=\{e\}$ and $e$ is a loop of $M$. Thus, one has 
$Q(x,y,a,b,M) = yQ(x,y,a,b,M/_A)$

\item If $M|_A = U_{1,1}$, then $A=\{e\}$. 
One has to distinguish between two subcases:
\begin{itemize}
\item $e$ is a coloop of $M$. Then, by Lemma \ref{lm:coloop}, 
$M/e = M\backslash e$. Thus, one has $Q(x,y,a,b,M) = xQ(x,y,a,b,M|_A)$.
\item $e$ is a nonseparating point of $M$. 
\end{itemize}

\item If $M/_A = U_{0,1}$, one can denote the ground set of $M/_A$ by 
$ \{e\}$. 
There are again two subcases to be considered:
\begin{itemize}
\item $e$ is a loop of $M$, one has that $M|_A = M\backslash_{(E-A)} = M\backslash_{\{e\}} = M/e$. 
Then one has $Q(x,y,a,b,M) = yQ(x,y,a,b,M|_A)$.
\item $e$ is a nonseparating point of $M$, then one has $M|_A = M\backslash_{(E-A)} = M\backslash_{\{e\}}$
\end{itemize}
\end{itemize}
We now insert all of this in equation \eqref{int1}; this leads to three types of sums over 
some element $e$ of the ground set $E$, $e$ being a loop, a coloop or a 
nonseparating point:
\begin{align}
s^{|E|-1}[\sum_{e \in E: e \mbox{\tiny{ is a coloop}}} Q(x,y,a,b,M) + \sum_{e\in E: e \mbox{\tiny{ is a loop}}} Q(x,y,a,b,M) + \sum_{e \in E: e \mbox{\tiny{ is a regular element}}}Q(x,y,a,b,M)] \notag\\
\end{align}
This rewrites as
\begin{align}
|E|s^{|E|-1}Q(x,y,a,b,M) = \frac{d\beta}{ds}(M),
\end{align}
which completes the proof.
\end{proof}

We can now state the main result of this paper, the recipe theorem
specifying how to recover the matroid polynomial $Q$ as 
an evaluation of the Tutte polynomial $T_M$:

\begin{theorem}
One has:
\begin{equation}
 \label{eqrecipe}
Q(x,y,a,b,M) = a^{n(M)}b^{r(M)}T_M(\frac{x}{b},\frac{y}{a}).
\end{equation}
\end{theorem}
\begin{proof}
 The proof is a direct consequence of Propositions \ref{alpha-tutte}, \ref{propalpha} and 
\ref{propbeta} and of Lemma \ref{lemamor}. This comes from the fact that 
one can apply the automorphism $\phi$ defined in \eqref{defphi}
to the differential equation 
\eqref{diffeqbeta}. One then obtains the differential equation 
\eqref{diffeqalpha} with modified parameters $x/b$ and $y/a$. 
Finally, the solution of this differential equation is (trivially) related to 
the matroid Tutte polynomial $T_M$ (see Proposition \ref{alpha-tutte})
and this concludes the proof. 
\end{proof}

\section{Conclusions}

We have thus proved in this paper the universality of the matroid Tutte polynomial
using differential equations of the same type as the Polchinski flow equation used 
in renormalization proofs in QFT. 
This analogy 
comes from the fact 
we differentiate with respect to two distinct type of edges of the graphs 
(see section \ref{sec:RG}).
The role of these two types of graph edges is played in the matroid case studied here by the
loop and the coloop type of elements in the matroid ground set.

As already announced in the Introduction, the matroid proofs given in this paper allow to 
prove the corresponding results for graphs. These graphs results were already conjectured in 
\cite{km}.

\bigskip

Let us end this paper by indicating as a possible direction for future work the 
investigation of the existence of a polynomial realization of matroid Hopf algebras.
This objective appears as particularly interesting 
because polynomial realizations of Hopf algebras 
substantially simplify the coproduct coassociativity check
(see, for example, 
the online version of the talk \cite{talk-jyt}). 
Such an example of polynomial realizations for the 
Hopf algebra of trees 
\cite{ck-0}
 (amongst other combinatorial 
Hopf algebras) was given in \cite{pol-real}.
Let us also mention that the task of finding polynomial realizations 
for matroid Hopf algebras appears to us to be close to the one of 
finding polynomial realizations 
for graph Hopf algebras, since both these algebraic combinatorial 
structures are based on the selection/contraction 
rule stated in the Introduction.

\bibliographystyle{abbrvnat}
\bibliography{FPSAC-DKHNT-final}
\label{sec:biblio}

\end{document}